\documentclass[11pt]{amsart}
\usepackage{amssymb, colordvi}
\usepackage{multirow}
\usepackage{graphicx,color}
\numberwithin{equation}{section}
\newtheorem{theorem}{Theorem}
\numberwithin{theorem}{section}

\newtheorem{proposition}[theorem]{Proposition}

\newtheorem{lemma}[theorem]{Lemma}

\newcommand{\C}{\mathbb{C}}
\newcommand{\calN}{\mathcal{N}}
\newcommand{\calNI}{\overset{\circ}{\mathcal{N}}}

\newcommand{\R}{\mathbb{R}}

\title[Real intersection points of a sparse plane curve with a line]{A sharp bound on the number of real intersection points of a sparse plane curve with a line}

\author{Fr\'ed\'eric Bihan}
\address{Laboratoire de Math\'ematiques\\
         Universit\'e Savoie Mont Blanc\\
         73376 Le Bourget-du-Lac Cedex\\
         France}

\email{Frederic.Bihan@univ-smb.fr}
\urladdr{http://www.lama.univ-savoie.fr/~bihan/}

\author{Boulos El Hilany}
\address{Laboratoire de Math\'ematiques\\
         Universit\'e Savoie Mont Blanc\\
         73376 Le Bourget-du-Lac Cedex\\
         France}

\email{boulos.el-hilany@univ-smb.fr}

\usepackage{latexsym,stmaryrd}

\begin{document}
\maketitle
\begin{abstract}

We prove that the number of real intersection points of a real line  with a real plane curve defined  by a polynomial with at most $t $ monomials is either infinite or does not exceed $6t -7$. This improves a result by M. Avendano. Furthermore, we prove that this bound is sharp for $t = 3$ with the help of Grothendieck's dessins d'enfant.
\end{abstract}
\section{Introduction}
The problem of estimating the number of real solutions of a system of polynomial equations is ubiquitous in mathematics and has obvious practical motivations. Fundamental notions like the degree or mixed volume give good estimates for the number of complex solutions of polynomial systems. However, these estimates can be rough for the number of real solutions when the equations have few monomials or a special structure (see~\cite{S}). In the case of a non-zero single polynomial in one variable, this is a consequence of Descartes' rule of signs which implies that the number of real roots is bounded by $2t-1$, where $t$ is the number of non-zero terms of the polynomial. Generalizations of Descartes' bound for polynomial and  more general systems have been obtained by A. Khovanskii~\cite{Kho}. The resulting bounds for polynomial systems have been improved by F. Bihan and F. Sottile~\cite{BS}, but still very few optimal bounds are known, even in the case of two polynomial equations in two variables. Polynomial systems in two variables where one equation has three non-zero terms and the other equation has $t$ three non-zero terms have been studied by T.Y. Li,  J.-M. Rojas and X. Wang~\cite{LRW}. They showed that such a system, allowing real exponents, has at most $2^t-2$ non-degenerate solutions contained in the positive orthant. This exponential bound has recently been refined into a polynomial one by P. Koiran, N. Portier and S. Tavenas \cite{KPT}.
The authors of~\cite{LRW} also showed that for $t=3$ the sharp bound is five.  Systems of two trinomial equations with five non-degenerate solutions in the positive orthant are in a sense rare~\cite{DRRS}.  Later M. Avenda\~no~\cite{A} considered systems of two polynomial equations in two variables, where the first equation has degree one and the other equation has $t$ non-zero terms. He showed that such a system has either an infinite number of real solutions or at most $6t-4$ real solutions. Here all solutions are counted with multiplicities, with the exception of the solutions on the real coordinate axis which are counted at most once. This reduces to counting the number of real roots of a polynomial $f(x,ax+b)$, where $a,b \in \R$ and $f \in \R[x,y]$ has at most $t$ non-zero terms. The question of optimality was not adressed in~\cite{A} and this was the motivation for the present paper. We prove the following result.

\begin{theorem} \label{MainTh}
Let $f \in \R[x,y]$ be a polynomial with at most three non-zero terms and let $a,b$ be any real numbers. Assume that the polynomial $g(x)=f(x,ax+b)$ is not identically zero. Then $g$ has at most $6t - 7$ real roots counted
with multiplicities except for the possible roots 0 and $-a/b$ that are counted at most once.
\end{theorem}

At a first glance this looks as a slight improvement of the main result of ~\cite{A}.
In fact, our bound is optimal at least for $t=3$.

\begin{theorem} \label{T:optimal}
The maximal number of real intersection points of a real line with a real plane curve defined by a polynomial with three non-zero terms is eleven.
\end{theorem} 

Explicitly,  the real curve with equation
\begin{equation}\label{E:equation}
-0,002404 \, xy^{18}+29 \, x^6y^3+x^3y=0
\end{equation}
intersects the real line $y=x+1$ in precisely eleven points in $\R^2$.

The strategy to construct this example is first to deduce from the proof of Theorem \ref{MainTh} some necessary conditions on the monomials of the desired equation.
Then, the use of real Grothendieck's dessins d'enfant~\cite{B,Br,O} helps to test the feasibility of certain monomials. Ultimately, computer experimentations lead to the
precise equation \eqref{E:equation}.

\section{Preliminary results}
We present some results of M. Avenda\~no~\cite{A} and add other ones.
Consider  a non-zero univariate polynomial $f(x) = \sum_{i=0}^d a_ix^i$ with real coefficients. Denote by
$V(f)$ the number of change signs in the ordered sequence $(a_0,\ldots,a_d)$ disregarding the zero terms. Recall that the famous Descartes' rule of signs asserts that
the number of (strictly) positive roots of $f$ counted with multiplicities does not exceed $V(f)$.
\smallskip

\begin{lemma}{~\cite{A}} \label{L:V}
We have $V((x+1)f)\leq\ V(f)$. 
\end{lemma}

The following result is straighforward.
\smallskip

\begin{lemma}~\cite{A}\label{AvRem}
If $f, g\ \in \mathbb{R}[x]$ and $g$ has $t$ terms, then $V(f+g)\leq V(f) + 2t$.
\end{lemma}
Denote by $\calN (h)$ the Newton polytope of a polynomial $h$ and by $\calNI (h)$ the interior of $\calN (h)$.
\smallskip

\begin{lemma}
\label{MyTh1}
If $f, g\ \in \mathbb{R}[X]$, $g$ has $t$ terms and $V(f+g)=V(f) + 2t$, then
$\calN (g)$ is contained in $\calNI (f)$.
\end{lemma}
 \begin{proof}
Assume that $\calN (g)$ is not contained in $\calNI (f)$. Writing $f(x)  = \sum_{i=1}^s{a_ix^{\alpha_i}}$ and $g(x) = \sum\limits_{j=1}^t{b_jx^{\beta_j}}$ with
$0\leq\alpha_1 < \cdots< \alpha_s$ and $0\leq \beta_1< \cdots < \beta_t$, we get $\beta_1 \leq \alpha_1$ or $\alpha_s \leq \beta_t$.
Assume that $\beta_1 \leq \alpha_1$ (the case $\alpha_s \leq \beta_t$ is symmetric). Then, obviously
$$V(f(x)+g(x)) \leq 1+V(f(x)+g(x)-b_1x^{\beta_1}).$$
By Lemma \ref{AvRem} we have
$$V(f(x)+g(x)-b_1x^{\beta_1}) \leq V(f)+2(t-1).$$
All together this gives $V(f+g) \leq 1+V(f)+2(t-1)=V(f)+2t-1$.
\end{proof}
\smallskip

\begin{proposition}{~\cite{A}}
\label{AvMainRes}
If $f\ \in \mathbb{R}[x,y]$ has $t$ non-zero terms, then
$$V(f(x,x+1))\leq 2t-2.$$
\end{proposition}
\begin{proof}
Write $f(x,y)=\sum_{k=1}^{n}{a_{k}(x)y^{\alpha_k}}$, with
$ 0 \leq\alpha_1< \cdots <\alpha_n$ and $\ a_k(x)\in\mathbb{R}[x]$. Denote by $t_{k}$ the number of non-zero terms of $a_k(x)$.
Define
$$f_k(x,y) = \sum_{j=k}^n a_j(x)y^{\alpha_j-\alpha_k} \, , \; k=1,\ldots,n,$$
and $f_{n+1}=0$. Then $f_k(x,x+1)=(x+1)^{\alpha_{k+1}-\alpha_k}f_{k+1}(x,x+1)+a_k(x)$
for $k=1, \ldots, n-1$ and $f_n(x,x+1)=a_n(x)$. Therefore, $V(f_k(x,x+1))\leq V(f_{k+1}(x,x+1)) + 2t_{k}$
by Lemma \ref{L:V} and Lemma \ref{AvRem}. Finally,
$V(f(x,x+1))\leq V(f_1(x,x+1))$
since
$f(x,x+1)=(x+1)^{\alpha_1}f_1(x,x+1)$.
We conclude that
$V(f(x,x+1)))\leq -2 + 2(t_{1}+\cdots+t_{n}) =2t-2$.
%
%
%
%
%
\end{proof}
%
%
%
%
%
\smallskip

\begin{proposition} \label{MyTh}
Let $f\ \in \mathbb{R}[x,y]$ be a polynomial with $t$ non-zero terms. Write it as
$
f(x,y)
=
\sum_{i=1}^t b_ix^{\beta_i}y^{\gamma_i}$
with $0 \leq \gamma_1 \leq \gamma_2 \leq \cdots \leq \gamma_t$.
If $V(f(x,x+1))=2t-2$, then
$$\calN(b_ix^{\beta_i}(x+1)^{\gamma_i}) \subset \calNI(b_tx^{\beta_t}(x+1)^{\gamma_t})$$
 (in other words,
$ \beta_t < \beta_i \leq \beta_i+\gamma_i < \beta_t+\gamma_t$) for $i=1,\ldots,t-1$.
%
%
%
%
\end{proposition}
\begin{proof}
We use the proof of Proposition~\ref{AvMainRes} keeping its notations.
Write $f(x,y)=\sum_{k=1}^{n}{a_{k}(x)y^{\alpha_k}}$ with
$ 0 \leq\alpha_1< \cdots <\alpha_n$ and assume that $V(f(x,x+1))=2t-2$. It follows from the proof of Proposition~\ref{AvMainRes} that
\begin{equation}\label{E:sharp}
V(f_k(x,x+1))= V(f_{k+1}(x,x+1)) + 2t_{k} \,  , \quad k=1,\ldots,n.
\end{equation}
Recall that $f_k(x,x+1)=(x+1)^{\alpha_{k+1}-\alpha_k}f_{k+1}(x,x+1)+a_k(x)$ for $k \leq n-1$. By Lemma \ref{MyTh1} and \eqref{E:sharp} we get
$\calN(a_k(x)) \subset \calNI((x+1)^{\alpha_{k+1}-\alpha_k}f_{k+1}(x,x+1))$
and thus
\begin{equation} \label{NewtPol1}
\calN(a_k(x)(x+1)^{\alpha_{k}}) \subset \calNI((x+1)^{\alpha_{k+1}}f_{k+1}(x,x+1)) 
\end{equation}
for $k=1,\ldots,n-1$.
We now show by induction on $n-k \geq 1$ that
\begin{equation} \label{NewtPol5}
\calNI((x+1)^{\alpha_{k+1}}f_{k+1}(x,x+1))\ \subset \calNI(a_n(x)(x+1)^{\alpha_n}).
\end{equation}
Together with \eqref{NewtPol1} this will imply $\calN(a_k(x)(x+1)^{\alpha_{k}}) \subset
\calNI(a_n(x) (x+1)^{\alpha_{n}})$ for $k=1,\ldots,n-1$, and thus $\calN(b_ix^{\beta_i}(x+1)^{\gamma_i}) \subset \calNI(b_tx^{\beta_t}(x+1)^{\gamma_t})$ for $i=1,\ldots,t-1$.
For $n-k=1$ the inclusion \eqref{NewtPol5} is obvious. Since $f_k(x,x+1)=(x+1)^{\alpha_{k+1}-\alpha_k}f_{k+1}(x,x+1)+a_k(x)$ and
$\calN(a_k(x)) \subset \calNI((x+1)^{\alpha_{k+1}-\alpha_k}f_{k+1}(x,x+1))$, we get
$\calNI(f_k(x,x+1))=\calNI((x+1)^{\alpha_{k+1}-\alpha_k}f_{k+1}(x,x+1))$.
Assuming \eqref{NewtPol5} is true for $k$ (hypothesis induction), this immediately gives 
$\calNI((x+1)^{\alpha_{k}}f_k(x,x+1)) \subseteq \calNI(a_n(x)(x+1)^{\alpha_{n}})$
and thus \eqref{NewtPol5} is proved for $k-1$.
\end{proof}

\section{Proof of Theorem~\ref{MainTh}}

We first recall the proof of the bound $6t-4$ in~\cite{A}.
Let $f(x,y)= \sum_{i=1}^{t}{b_ix^{\beta_i}y^{\gamma_i}}\in\mathbb{R}[x,y]$
be a polynomial with at most $t$ non-zero terms, and let $a$, $b$ $\in \mathbb{R}$. Set
$g(x)=f(x,ax+b)$.  If $a=0$ or $b=0$, then $f$ has at most $t$ non-zero
 terms and Descartes' rule of signs implies that either $g=0$ or $g$ has at most $2t-1\leq 6t-4$ real roots (counted with multiplicities 
except for the possible root $0$). If $ab \neq 0$, then the real roots of 
$f(x,ax+b)$ correspond bijectively to the real roots of
$f(bx/a,b(x+1))=\hat{f}(x,x+1)$, where
$\hat{f}(x,y)=\sum_{i=1}^{t}{b_ia^{-\beta_i}b^{\beta_i+\gamma_i}x^{\beta_i}y^{\gamma_i}}$.
Since this bijection preserves multiplicities and maps the possible roots $0$ 
and $-b/a$ of $g$ to the roots $0$ and $-1$ of $\hat{f}(x,x+1)$, it suffices to consider the case 
$a=b=1$, i.e. $g(x)=f(x,x+1)$. So we now consider $g(x)=f(x,x+1)$.
Assume that $g\neq 0$ and denote by $d$ the degree of $g$.

Descartes' rule of signs and  Proposition~\ref{AvMainRes} imply that the number of positive roots of $g$ counted with 
multiplicities is at most $2t-2$.
The roots of $g$ in $]-\infty ,-1[$ correspond bijectively to the positive roots of $g(-1-x)=f(-1-x,-x)=\sum_{i=1}^{t}{b_i(-1)^{\beta_i+\gamma_i}x^{\gamma_i} (x+1)^{\beta_i}}$.
Therefore,  by Proposition~\ref{AvMainRes}  the number of roots (counted with 
multiplicities) of $g$ in $]-\infty , -1[$ cannot exceed $2t-2$.
Finally, the roots of $g$ in 
$]-1,0[$ correspond bijectively to the positive roots of $ (x+1)^{d}g(\frac{-x}{x+1})=
(x+1)^{d}f(\frac{-x} {x+1},\frac{1}{x+1})=\sum_{i=1}^{t}{b_i(-1)^{\beta_i}x^{\beta_i}(x+1)^{d-\beta_i-\gamma_i}}$.
%
Thus, by Proposition~\ref{AvMainRes} there are at most $2t-2$ such roots.
All together, this leads to the conclusion that $g$ has at most $3(2t-2)+2=6t -4$ real roots counted with multiplicities except for the possible roots 0 and $-1$ that are counted at most once.
\smallskip

We now start the proof of Theorem \ref{MainTh}. 
Set $I_1=]0,+\infty[$, $ I_2=]-\infty,-1[$ and $I_3=]-1,0[$.
%
%
For $h \in \R[x]$ define
$$V_{I_1}(h)=V(h) \, , \quad V_{I_2}(h)=V(h(-1-x)) \quad \mbox{and}$$
$$ 
V_{I_3}(h)=V((x+1)^{\deg(h)}h(\frac{-x} {x+1})).$$

By Descartes' rule of signs the number of roots of $h$ in $I_i$ does not exceed $V_{I_i}(h)$. To prove Theorem \ref{MainTh}, it suffices to show that
\begin{equation} \label{E:toshow}
V_{I_1}(g)+V_{I_2}(g)+V_{I_3}(g) \leq 3(2t-2)-3
\end{equation}
Define polynomials
$$ h_1(x)=x^dh(\frac{1}{x}) \; , \quad h_2(x)=(x+1)^dh(\frac{-x}{x+1})
\quad \mbox{and}  \quad h_3(x)=h(-1-x)
$$ 
so that $V_{I_1} (h_1)=V_{I_1}(h)$, $V_{I_1} (h_2)=V_{I_3}(h)$ and $V_{I_1} (h_3)=V_{I_2}(h)$.
\smallskip

\begin{lemma}\label{L:symmetry}
For any $i,j,k$ such that $\{i,j,k\}=\{1,2,3\}$, we have
$$V_{I_i}(h_i)=V_{I_i}(h) \quad\mbox{and} \quad V_{I_i}(h_j)=V_{I_k}(h)$$
\end{lemma}
\begin{proof}
We have $h_1(-x-1)=(-1)^d (x+1)^d h(-\frac{1}{x+1})$. Thus
$V(h_1(-x-1))=V((x^{-1}+1)^d h(-\frac{1}{x^{-1}+1}))=V((\frac{x+1}{x})^dh(-\frac{x}{x+1}))=
V((x+1)^dh(-\frac{x}{x+1}))$, and we get $V_{I_2}(h_1)=V_{I_3}(h)$.
We have $(x+1)^d h_1(-\frac{x}{x+1})=(-x)^dh(-1-x^{-1})$ from which we obtain $V_{I_3}(h_1)=V_{I_2}(h)$.
\smallskip 

Equalities $V_{I_2}(h_2)=V_{I_2}(h)$ and $V_{I_3}(h_2)=V_{I_1}(h)$ follow from $h_2(-1-x)=(-x)^d h(-1-x^{-1})$
and $(x+1)^dh_2(-\frac{x}{x+1})=h(x)$.
\smallskip

Finally, $V_{I_2}(h_3)=V_{I_1}(h)$ comes from $h_3(-x-1)=h(x)$ and
$V_{I_3}(h_3)=V_{I_3}(h)$ is a consequence of $(x+1)^dh_3(-\frac{x}{x+1})=(x+1)^d h(-\frac{1}{x+1})$ and the equality
$V((x+1)^d h(-\frac{1}{x+1}))=V_{I_3}(h)$ shown above. 
\end{proof}

We now proceed to the proof of \eqref{E:toshow}.
We already know that $V_{I_i}(g) \leq 2t-2$ for $i=1,2,3$.
If $V_{I_i}(g) \leq 2t-3$ for all $i$, then \eqref{E:toshow} is trivially true.
With the help of Lemma \ref{L:symmetry}, it suffices now to show that if $V_{I_1}(g)=2t-2$ then
$V_{I_2}(g) \leq 2t-3$, $V_{I_3}(g) \leq 2t-3$, and $V_{I_2}(g)+V_{I_3}(g)< 2(2t-3)$.
So assume $V_{I_1}(g)= 2t-2$. Then by Proposition \ref{MyTh}
\begin{equation}\label{E:useful}
\beta_t < \beta_i \leq \beta_i+\gamma_i < \beta_t+\gamma_t,\;   , \quad  i=1,\ldots,t-1.
\end{equation}
We have $g(-1-x)=\sum_{i=1}^{t}
{b_i(-1)^{\beta_i+\gamma_i}x^{\gamma_i}(x+1)^{\beta_i}}$.
Recall that $V_{I_2}(g)=V(g(-x-1)) \leq 2t-2$ by Proposition \ref{AvMainRes}.
From \eqref{E:useful}, we get $\gamma_t >\gamma_i$ for $i=1,\ldots,t-1$. It follows then from Proposition \ref{MyTh} that $V(g(-x-1)) \leq 2t-3$.

Write
$g(-1-x)=\tilde{g}(-x-1)+ b_t(-1)^{\beta_t+\gamma_t}x^{\gamma_t}(x+1)^{\beta_t}$, and then
$g(-1-x)(x+1)^{-\beta_t}=\tilde{g}(-x-1)(x+1)^{-\beta_t}+ b_t(-1)^{\beta_t+\gamma_t} x^{\gamma_t}$.
We note that \eqref{E:useful} implies $\beta_t < \beta_i$ for $i=1,\ldots,t-1$, so that both members of the previous equality are polynomials.
Moreover, from \eqref{E:useful} we also get $\beta_i-\beta_t+\gamma_i < \gamma_t$, and thus $\gamma_t$ does not belong to the Newton polytope of
the polynomial $\tilde{g}(-x-1)(x+1)^{-\beta_t}$. It follows that $V(g(-1-x)(x+1)^{-\beta_t}) \leq V(\tilde{g}(-x-1)(x+1)^{-\beta_t})+1$.
By Lemma \ref{L:V} we have $V(g(-1-x)) \leq V(g(-x-1)(x+1)^{-\beta_t})$. Therefore, $V(g(-1-x)) \leq V(\tilde{g}(-x-1)(x+1)^{-\beta_t})+1$.
On the other hand Proposition \ref{AvMainRes} yields $V(\tilde{g}(-x-1)(x+1)^{-\beta_t}) \leq 2(t-1)-2=2t-4$.

Therefore, if $V(g(-1-x))=2t-3$, then $V(\tilde{g}(-x-1)(x+1)^{-\beta_t})=2t-4$, and we may apply Proposition \ref{MyTh} to $\tilde{g}(-x-1)(x+1)^{-\beta_t}$
in order to get

 \begin{equation}\label{i0}
\gamma_{i_0} < \gamma_i \leq \gamma_i+\beta_i < \gamma_{i_0}+\beta_{i_0} \; \mbox{for all} \; i=1,\ldots,t-1 \; \mbox{and} \; i \neq i_0,
\end{equation}
where $i_0$ is determined by $\beta_{i_0} \geq \beta_i$ for $i=1,\ldots,t-1$.

%
%
%
%

Starting with $g_1(x)=x^d g(1/x)=\sum_{i=1}^t b_i x^{d-\beta_i-\gamma_i}(x+1)^{\gamma_i}$ instead of $g$ in the previous computation, we obtain that
if $V(g_1)=2t-2$ then $V_{I_2}(g_1) \leq 2t-3$ and if
$V_{I_2}(g_1)=2t-3$, then the substitution of $d-\beta_i-\gamma_i$ for $\beta_i$ in \eqref{i0} holds true:
\begin{equation}\label{i1}
\gamma_{i_1} < \gamma_i \leq d-\beta_i < d-\beta_{i_1} \; \mbox{for all} \; i=1,\ldots,t-1 \; \mbox{and} \; i \neq i_1,
\end{equation}
where $i_1$ is determined by $d-\beta_{i_1}-\gamma_{i_1} \geq d-\beta_i-\gamma_i$ for $i=1,\ldots,t-1$. 

On the other hand, $V(g)=V(g_1)$ and $V(g_1(-x-1))=V_{I_2}(g_1)=V_{I_3}(g)$ by Lemma \ref{L:symmetry}.
Thus if $V(g)=2t-2$ then $V_{I_3}(g) \leq 2t-3$ and if $V_{I_3}(g)=2t-3$, then formula \eqref{i1} holds true.
It turns out that \eqref{i0} and \eqref{i1} are incompatible. 
Indeed, if \eqref{i0} and \eqref{i1} hold true simultaneously, then $i_0=i_1$ but then \eqref{i1} implies that $\gamma_{i_0}+\beta_{i_0} < \gamma_i+\beta_i $
for all $1,\ldots,t-1$ with $i \neq i_0$ which contradicts \eqref{i0}.
Consequently, if $V(g)=V_{I_1}(g)=2t-2$, then $V_{I_2}(g) \leq 2t-3$, $V_{I_3}(g) \leq 2t-3$ and $V_{I_2}(g) +V_{I_3}(g) <2(2t-3)$. 

\section{Optimality}

We prove that the bound in Theorem \ref{MainTh} is sharp for $t=3$ (Theorem \ref{T:optimal}).
We look for
%
%
a polynomial $P \in \R[x,y]$ with three non-zero terms such that $P(x,x+1)$ has nine real roots distinct from $0$ and $-1$. It follows from the previous section that
if such $P$ exists then, either $P(x,x+1)$ has three roots in each interval $I_1$, $I_2$ and $I_3$, or $P(x,x+1)$ has four roots in one interval, three roots in another interval, and two roots in the last one. We give necessary conditions for the second case, which thanks to Lemma \ref{L:symmetry} reduces to the case where
$P(x , x + 1)$ has four roots in $I_1=]0,+\infty[$, three roots in $I_3=]-1 ,0[$ and two roots in $I_2=]-\infty,-1[$.

Multiplication of $P$ by a monomial does not alter the roots of $P(x,x+1)$ in $\R \setminus \{0,-1\}$, so dividing by the smallest power of $x$,
we may assume that $P$ has the following form
$$ P(x,y) = ay^{l_1} + bx^{k_2}y^{l_2} + x^{k_3}y^{l_3},$$
where $k_2$, $k_3$, $l_1$, $l_2$, $l_3$ are nonnegative integer numbers and $a,b$ are real numbers.
\smallskip

 \begin{lemma} \label{L:tech1}
If $P(x,x + 1)$ has four real positive roots, then $k_2 >0$,  $k_3 >0$, $ l_1 >l_2 + k_2$ and
$l_1>l_3 + k_3$.
 \end{lemma} 
  \begin{proof}
If $P(x,x + 1)$ has four real positive roots, then $V(P(x , x + 1)) = 4$.
Rewriting $P(x , x + 1) =  \sum_{i=1}^3 b_ix^{\beta_i}(x + 1)^{\gamma_i}$ with $0 \leq \gamma_1 \leq 
 \gamma_2 \leq \gamma_3$, Proposition~\ref{MyTh} yields $\beta_3 < \beta_i \leq \beta_i + \gamma_i < \beta_3 + \gamma_3$ for 
  $i = 1,2$. Since $k_2$ and $k_3$ are nonnegative, we get $\beta_3 = 0$, $k_2,k_3 >0$ and $\beta_3 + \gamma_3 = \gamma_3 = l_1$, so $l_1 >\max (l_2 + k_2 , l_3 + k_3)$.
 \end{proof}

Since $l_1>l_2$ and  $l_1 > l_3$,  we may divide $P(x,x+1)$ by $(x+1)^{l_2}$ or $(x+1)^{l_3}$ to get a polynomial equation  with the same solutions
in $\R \setminus \{0,-1\}$. So without loss of generality we may assume that
 \begin{equation} \label{Preduced}
 P(x,x+1)=     a(x+1)^{l_1} + bx^{k_2}(x+1)^{l_2} + x^{k_3},
   \end{equation}   
 where $k_2, k_3 >0$, $l_2 \geq 0$, $l_1>k_2+l_2$ and $l_1 >k_3$.
\smallskip

\begin{lemma} \label{NPlem}
Assume that the polynomial \eqref{Preduced} has four roots in $I_1$,  and three roots in $I_3$ or $I_2$. Then $k_3$
 does not belong to the interval $[k_2,k_2+l_2]$. Moreover, we have $a<0$ and $b>0$.
\end{lemma}
\begin{proof}
We prove that if $k_2 \leq k_3 \leq k_2 + l_2$, then \eqref{Preduced} 
has at most two roots in $I_2$ and in $I_3$.

The roots in $I_2$ are in bijection with the positive roots of
$$\displaystyle P(-x -1 , -x) = (-1)^{l_1}ax^{l_1} + (-1)^{k_2 + l_2}bx^{l_2}(x + 1)^{k_2} + 
   (-1)^{k_3}(1 + x)^{k_3}.$$
Recall that $l_2 \geq 0$. If  $k_2 \leq k_3 \leq k_2+l_2$ then 
Proposition \ref{MyTh} yields $V((-1)^{k_2 + l_2}bx^{l_2}(x + 1)^{k_2} + (-1)^{k_3}(1 + x)^{k_3}) \leq 1$.
Now, since $l_1 > k_2 + l_2$ and $l_1 >k_3$, we get $V( P(-x -1 , -x)) \leq 2$, and thus \eqref{Preduced} has at most two roots in $I_2$.

The roots in $I_3$ are in bijection with the positive roots of
$$(1 + x)^{l_1}P(\frac{-x}{x + 1} , \frac{-x}{x + 1} + 1 )=
 a + b(-1)^{k_2}x^{k_2}(1 + x)^{l_1 - k_2 -l_2} + 
 (-1)^{k_3}x^{k_3}(1 + x)^{l_1 - k_3}$$
 From $k_3 \leq k_2+l_2$, we get $l_1-k_2-l_2 \leq l_1-k_3$. Thus, Proposition \ref{MyTh} together with $k_2 \leq k_3$ yields
$V(b(-1)^{k_2}x^{k_2}(1 + x)^{l_1 - k_2 -l_2} + 
    \displaystyle (-1)^{k_3}x^{k_3}(1 + x)^{l_1 - k_3})$\ $\leq 1$.
From $k_2,k_3>0$ we get $V((1 + x)^{l_1}P(\frac{-x}{x + 1} , \frac{-x}{x + 1} + 1 ))$ \ $\leq 2$, and thus \eqref{Preduced} has at most two roots in $I_3$.

Finally, if \eqref{Preduced}  has four positive roots, then obviously $ab<0$.
If $k_3$ does not belong to $[k_2,k_2+l_2]$ and $a>0$, then
$V((x+1)^{l_1} + bx^{k_2}(x+1)^{l_2} + x^{k_3})=V((x+1)^{l_1} + bx^{k_2}(x+1)^{l_2})$ (recall that $k_2 \leq k_2+l_2 < l_1$). But the second sign variation is a most two by Proposition \ref{AvMainRes}. We conclude that $a<0$ and $b>0$.
\end{proof}
\smallskip

\begin{lemma}\label{L:parity}
Assume that the polynomial \eqref{Preduced} has four roots in $I_1$,  two roots in $I_2$ and three roots in $I_3$.
Assume furthermore that $k_3 < k_2$.
Then, $l_1$ is odd, $k_2$ is odd, $k_3$ is even and $l_2$ is even.
\end{lemma}
\begin{proof}
Since \eqref{Preduced}  has exactly nine real roots counted with multiplicity, its degree $l_1$ is odd.
We have already seen that if \eqref{Preduced} has four roots in $I_1=]0,+\infty[$,  two roots in $I_2=]-\infty,-1[$ and three roots in $I_3=]-1 ,0[$,
then $a<0$, $b>0$, $l_1>l_2 $ and $k_3 \notin [k_2,k_2+l_2]$. Assume from now on that $k_3 < k_2$.

Since \eqref{Preduced} has two roots in $I_2=]-\infty,-1[$, we have
$V(P(-x -1 , -x)) \geq 2$, where $P(-x -1 , -x)=(-1)^{k_3}(1 + x)^{k_3} + (-1)^{k_2 + l_2}bx^{l_2}(x + 1)^{k_2} +  (-1)^{l_1}ax^{l_1}$.
But since $k_3 < k_2 \leq k_2+l_2 < l_1$, we get that $(-1)^{k_3} \cdot (-1)^{k_2 + l_2}b<0$ and $(-1)^{k_2 + l_2}b \cdot (-1)^{l_1}a <0$.
Using $a<0$ and $b>0$, we obtain that $k_2+l_2$ is odd and $k_3$ is even.

Since \eqref{Preduced} has three roots in $I_3=]-1 ,0[$, we have $V((1 + x)^{l_1}P(\frac{-x}{x + 1} , \frac{-x}{x + 1} + 1 )) \geq 3$, where
$(1 + x)^{l_1}P(\frac{-x}{x + 1} , \frac{-x}{x + 1} + 1 )=  a + b(-1)^{k_2}x^{k_2}(1 + x)^{l_1 - k_2 -l_2} + 
 (-1)^{k_3}x^{k_3}(1 + x)^{l_1 - k_3 -l_3}$. We know that $k_3$ is even and that $b>0$. Thus in order to get coefficients with different signs in
$b(-1)^{k_2}x^{k_2}(1 + x)^{l_1 - k_2 -l_2} + 
 (-1)^{k_3}x^{k_3}(1 + x)^{l_1 - k_3 -l_3}$, the integer $k_2$ should be odd. Since we know that $k_2+l_2$ is odd, this gives that $l_2$ is even.
 \end{proof}

Assume now that \eqref{Preduced} has four roots in $I_1$,  two roots in $I_2$ and three roots in $I_3$.
Then $a<0$, $b>0$ and $k_3$ does not belong to $[k_2,k_2+l_2]$ by Lemma \ref{NPlem}.
Assume that $k_3 < k_2$. Then $l_1$ is odd, $k_2$ is odd, $k_3$ is even and $l_2$ is even by Lemma \ref{L:parity}.
The roots of \eqref{Preduced} are solutions to the equation
$f(x)=-a$, where $\displaystyle f(x)= bx^{k_2}(1 + x)^{l_2 - l_1} + x^{k_3}(1 + x)^{- l_1}$. Since the rational function $f$ has no pole outside $\{-1,0\}$, by Rolle's Theorem
its derivative has at least three roots in $I_1$, one root in $I_2$ and two roots in $I_3$. We compute that $f'(x)=0$ is equivalent to $\Phi(x)=1$, where $\Phi$ is the rational map
\begin{equation}\label{E:fraction}
\Phi(x)=\frac{-bx^{k_2 - k_3}(1 + x)^{l_2}A_1(x)}{A_2(x)},
\end{equation}
with $A_1(x)=(k_2 + l_2 - l_1)x+k_2$ and $A_2(x)=(k_3 - l_1)x+k_3$.
From $0 <k_3<k_2$, $l_2 \geq 0$ and $l_1>0$, we obtain that
the roots of $A_1$ and $A_2$ satisfy  $0 < \frac{k_3}{l_1 - k_3} < \frac{k_2}{l_1 - k_2 - l_2}$. Moreover, 
the roots of $\Phi$ are $-1$ with even multiplicity $l_2$, $0$ with odd multiplicity $k_2-k_3$ and the positive root of $A_1$ (which is a simple root of $\Phi$).
The poles of $\Phi$ are the positive root of $A_2$ and the point at infinity which has multiplicity $\mbox{deg}(\Phi)-1$ if we homogeinize $\Phi$ into a rational map from the Riemann sphere $\C P^1$ to itself.

\begin{figure}[htbp]
\begin{center}
 \resizebox{.5\textwidth}{!}{
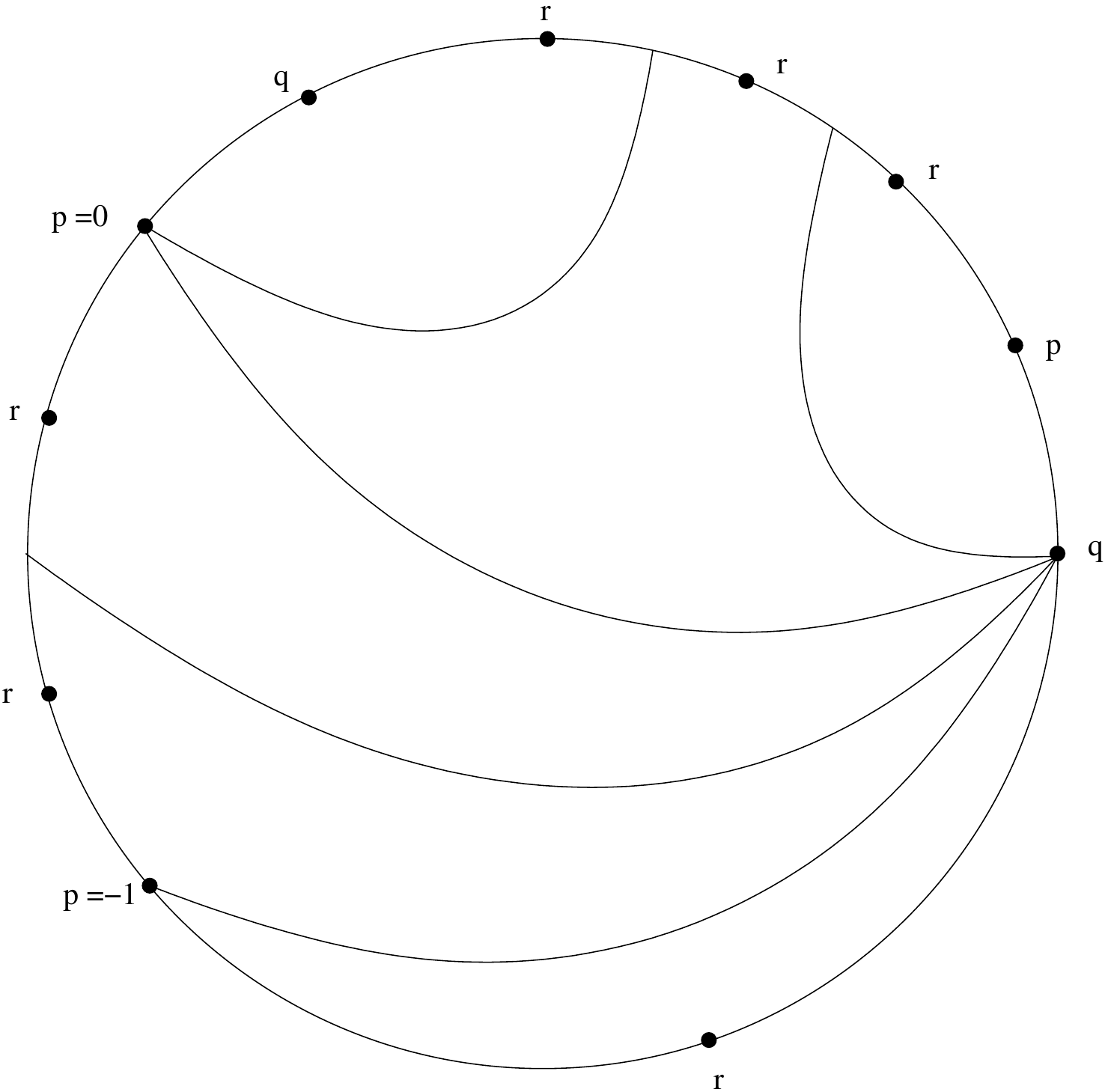
}
\caption{Real dessin d'enfant for $\varphi$.}
\label{F:dessin}
\end{center}
\end{figure}

We find exact values of coefficients and exponents of \eqref{E:fraction} in the following way.
Note that the exponents of \eqref{E:fraction} are independent of $l_1$. We first choose small values $k_2=5$, $k_3=2$, $l_2=2$ satisfying the above parity conditions.
Then, we look for a function 

\begin{equation}\label{E:candidate}
\varphi(x)=\frac{cx^3(x+1)^2(x-\rho_1)}{x-\rho_2},
\end{equation}
such that $c$ is some real constant, $0 <\rho_2 <\rho_1$ and $\varphi(x)=1$ has three solutions in $I_1$, one solution in $I_2$ and two solutions in $I_3$.

The existence of such a function $\varphi$ is certified by Figure \ref{F:dessin} thanks to Riemann Uniformization Theorem.
Figure \ref{F:dessin} represents the intersection of the graph $\Gamma=\varphi^{-1}(\R P^1)$ with one connected component of $\C P^1 \setminus \R  P^1$
(the whole graph can be recovered from this intersection since it is invariant by complex conjugation). Such a graph is called {\it real dessin d'enfant} (see~\cite{B,Br,O} for instance).
Each connected component of $\C P^1 \setminus \Gamma$ (a disc)
can be endowed with an orientation inducing the order $p<q<r$ for the three letters $p,q,r$ in its boundary so that two adjacent discs get opposite orientations.
Choose coordinates on the target space $\C P^1$. Choose one connected component of  $\C P^1 \setminus \Gamma$ and send it homeomorphically to one connected component of $\C P^1 \setminus \R P^1$ so that letters $p$ are sent to $(1:0)$, letters $q$ are sent to $(0:1)$ and letters $r$  to $(1:1)$.
Do the same for each connected component of $\C P^1 \setminus \Gamma$ so that the resulting homeomorphisms extend to an orientation preserving continuous map
$\varphi: \C P^1 \rightarrow \C P^1$. Note that two adjacent connected components of $\C P^1 \setminus \Gamma$
are sent to different connected components of $\C P^1 \setminus \R P^1$. The Riemann Uniformization Theorem implies that $\varphi$ is a real rational map for the standard complex structure
on the target space and its pull-back by $\varphi$ on the source space.
The degree of $\varphi$ is half the number of connected components of $\C P^1 \setminus \Gamma$ (a generic point on the target space has one preimage in each  component
of $\C P^1 \setminus \Gamma$ with the correct orientation). The critical points of $\varphi$ are the vertices of $\Gamma$, and the multiplicity of each critical point is half its valency.
The letters $p$ are the inverse images of $(1:0)$, the letters $q$ are the inverse images of $(0:1)$ and the letters $r$ are inverse images of $(1:1)$. 
In Figure \ref{F:dessin}, we see three letters $p$ on $\Gamma$, two of them being critical points with multiplicities three and two respectively
(the valencies are six and four, recall that Figure \ref{F:dessin} show only one half of $\Gamma$).
We also see two letters $q$, one of them being a critical point of multiplicity five.
Choose coordinates on the source space $\C P^1$ so that the critical point $p$ of multiplicity three has coordinates $(1:0)$, the other critical point $p$ has coordinates $(1:-1)$ and the critical point $q$ is the point with coordinates $(0:1)$. In standard affines coordinates (for both the source and the target spaces) of the chart where the first homogeneous coordinate does not vanish, any
rational map whose real dessin d'enfant is as depicted in Figure \ref{F:dessin} is defined by \eqref{E:candidate}.
From Figure \ref{F:dessin}, we see that $0 < \rho_2< \rho_1$ and that $\varphi$ has the desired number of inverse images (letters $r$) of $1$ in each interval
$I_i$.

Now we want to identify \eqref{E:candidate} and \eqref{E:fraction}.
Recall that $k_2=5$, $k_3=2$, $l_2=2$ are fixed. We look at the function $\frac{x^3(x+1)^2(x-\rho_1)}{x-\rho_2}$,
where $\rho_1=\frac{k_2}{l_1 - k_2 - l_2}$ and $\rho_2=\frac{k_3}{l_1 - k_3}$, and increase $l_1$  so that some level set of this function has three solutions in $I_1$,
one solution in $I_2$ and two solutions in $I_3$. It turns out that $l_1 = 17$ is large enough and the level set gives the value $29$ for $b$.
Finally, integrating $\Phi$ 
and choosing $a=-0,002404$, we get
$$
-0.002404(x+1)^{17}+29x^5(x+1)^2+x^2
$$
for \eqref{Preduced}. This polynomial has fours roots in $I_1$, two roots in $I_2$ and three roots in $I_3$.
This has been computed using SAGE version 6.6 which gives the following approximated roots:
$0.18859$, $0.22206$, $0.25196$, $0.44416$ in $I_1$,
$-3.96032$, $-1.15048$ in $I_2$, and
$-0.61459$, $-0.58528$, $-0.03594$ in $I_3$.

Multiplying this polynomial by $x(x+1)$ gives a polynomial of the form $P(x,x+1)$ (where $P \in \R[x,y]$ has three non-zero terms)
having eleven real roots.

%
 
 \providecommand{\bysame}{\leavevmode\hbox to3em{\hrulefill}\thinspace}
\providecommand{\MR}{\relax\ifhmode\unskip\space\fi MR }
\providecommand{\MRhref}[2]{%
  \href{http://www.ams.org/mathscinet-getitem?mr=#1}{#2}
}
\providecommand{\href}[2]{#2}

\end{document}